\documentclass[a4paper,draft,reqno,12pt]{amsart}
\usepackage[english]{babel}
\usepackage{amsmath}
\usepackage{amssymb}
\usepackage{amscd}
\usepackage{amsthm}
\usepackage{euscript}
\usepackage[all]{xy}
\newtheorem{proposition}{Proposition}
\newtheorem{conjecture}{Conjecture}
\newtheorem{problem}{Problem}
\newtheorem{lemma}{Lemma}
\newtheorem{theorem}{Theorem}

\newtheorem{corollary}{Corollary}
\theoremstyle{definition}
\newtheorem{definition}{Definition}
\newtheorem{example}{Example}
\theoremstyle{remark}
\newtheorem {remark}{Remark}

\DeclareMathOperator{\Spec}{Spec}
\DeclareMathOperator{\Hom}{Hom}
\DeclareMathOperator{\Aut}{Aut}
\DeclareMathOperator{\SAut}{SAut}
\DeclareMathOperator{\SL}{SL}

\DeclareMathOperator{\reg}{reg}
\DeclareMathOperator{\Sp}{Sp}
\DeclareMathOperator{\cone}{cone}
\DeclareMathOperator{\Cl}{Cl}

\DeclareMathOperator{\codim}{codim}

\DeclareMathOperator{\rk}{rk}

\def\Ker{{\rm Ker}}

\def\PGL{{\rm PGL}}

\def\GG{{\mathbb G}}

\def\KK{{\mathbb K}}
\def\TT{{\mathbb T}}
\def\ZZ{{\mathbb Z}}

\def\SS{{\mathbb S}}

\def\QQ{{\mathbb Q}}
\def\PP{{\mathbb P}}
\def\AA{{\mathbb A}}
\newcommand{\XX}{\mathcal{X}}
\newcommand{\YY}{\mathcal{Y}}
\def\VVV{\mathcal{V}}
\def\WWW{\mathcal{W}}
\def\NNN{\mathcal{N}}
\def\AAA{\mathcal{A}}
\def\BBB{\mathcal{B}}

\def\OOO{\mathcal{O}}

\def\RRR{\mathfrak{R}}

\begin{document}
\date{}
\title[Gale duality and homogeneous toric varieties]{Gale duality and homogeneous toric varieties}
\author{Ivan Arzhantsev}
\thanks{The research was supported by the grant RSF-DFG 16-41-01013}
\address{National Research University Higher School of Economics, Faculty of Computer Science, Kochnovskiy Proezd 3, Moscow, 125319 Russia}
\email{arjantsev@hse.ru}



\maketitle

\begin{abstract}
A non-degenerate toric variety $X$ is called $S$-homogeneous if the subgroup of the automorphism group $\Aut(X)$ generated by root subgroups acts on $X$ transitively. We~prove that maximal $S$-homogeneous toric varieties are in bijection with pairs $(P,\AAA)$, where $P$ is an abelian group and $\AAA$ is a finite collection of elements in $P$ such that $\AAA$ generates the group $P$ and for every $a\in\AAA$ the element $a$ is contained in the semigroup generated by $\AAA\setminus\{a\}$. We show that any non-degenerate homogeneous toric variety is a big open toric subset of a maximal $S$-homogeneous toric variety. In particular, every homogeneous toric variety is quasiprojective. We conjecture that any non-degenerate homogeneous toric variety is $S$-homogeneous.
\end{abstract}

\section{Introduction}

Let $X$ be an irreducible algebraic variety over an algebraically closed field $\KK$ of characteristic zero. The variety $X$ is said to be \emph{toric} if $X$ is normal and admits an effective action $T\times X\to X$ of an algebraic torus $T$ with an open orbit. It is well known that toric varieties $X$ are characterized by fans of polyhedral cones $\Sigma_X$ in the vector space $N\otimes_{\ZZ}\QQ$, where $N$ is the lattice of one-parameter subgroups in $T$, see \cite{De,Oda, Fu,CLS}.

The linear Gale duality as defined in \cite{OP,BH} (see also \cite[Section~2.2]{ADHL}) provides an alternative combinatorial language for toric varieties. It is developed systematically in \cite{BH} under the name of bunches of cones in the divisor class group. As is mentioned in the Introduction to~\cite{BH}, this approach gives very natural description of geometric phenomena connected with divisors. The aim of the present paper is to show that this language is natural also for toric varieties homogeneous under automorphism groups.

An algebraic variety $X$ is said to be \emph{homogeneous} if the automorphism group $\Aut(X)$ acts on $X$ transitively. The class of homogeneous varieties is wide. In particular, it includes all homogeneous spaces of algebraic groups. It is an important problem to classify homogeneous varieties among varieties of a given type. In this paper we are interested in homogeneous toric varieties.

Let $X$ be a toric variety and $S(X)$ the subgroup of $\Aut(X)$ generated by all root subgroups in $\Aut(X)$. It is well known that root subgroups are in one-to-one correspondence with Demazure roots of the fan $\Sigma_X$, see \cite{De,Oda,Cox}. Nowadays Demazure roots and their generalizations became a central tool in many research projects, see \cite{Nill,Li,AKZ,Ba,AB,AHHL,AK}. For instance, Demazure roots and Gale duality was used in \cite{Ba} and \cite{AB} to describe orbits of the group $\Aut(X)$ on complete and affine toric varieties $X$, respectively.

We say that a toric variety $X$ is \emph{$S$-homogeneous} if the group $S(X)$ acts transitively on $X$. An $S$-homogeneous toric variety $X$ is said to be \emph{maximal} if it does not admit a proper open toric embedding $X\subseteq X'$ into an $S$-homogeneous toric variety $X'$ with
${\codim_{X'}(X'\setminus X)\ge 2}$.

Consider an abelian group $P$ and a collection $\AAA=\{a_1,\ldots,a_r\}$ of elements of $P$ (possibly with repetitions) that generates the group $P$. Denote by $A$ the semigroup generated
by~$\AAA$. We say that a collection $\AAA$ is \emph{admissible} if for every $a_i\in\AAA$
the semigroup generated by $\AAA\setminus\{a_i\}$ coincides with $A$. A pair $(P,\AAA)$, where $P$ is an abelian group and $\AAA$ is an admissible collection of elements of $P$,
is said to be \emph{equivalent} to a pair $(P',\AAA')$, if there is an isomorphism of abelian groups $\gamma\colon P\to P'$ such that $\gamma(\AAA)=\AAA'$.

\smallskip

Our main result provides an elementary description of maximal $S$-homogeneous toric varieties.

\begin{theorem} \label{tmain}
There is a one-to-one correspondence between maximal $S$-homogeneous toric varieties and equivalence classes of pairs $(P,\AAA)$, where $P$ is an abelian group and $\AAA$ is an admissible collection of elements in $P$.
\end{theorem}

If $X$ is a toric variety corresponding to a pair $(P,\AAA)$, then the group $P$ is isomorphic to the divisor class group $\Cl(X)$ and the collection $\AAA$ coincides with the set of classes of $T$-invariant prime divisors $[D_1],\ldots,[D_r]$. In particular, the dimension of $X$ equals
$r-\rk(P)$.

\smallskip

Let us give an overview of the content of the paper. In Section~\ref{s1} we recall basic facts on toric varieties and associated fans. Section~\ref{s2} contains some background on Demazure roots and corresponding root subgroups. We define strongly regular fans and prove that a toric variety is $S$-homogeneous if and only if the associated fan is strongly regular (Proposition~\ref{psh}).

In Section~\ref{s3} we collect basic properties of the linear Gale duality. Section~\ref{s4} provides a modification of this duality that takes into account a lattice containing a vector configuration. We call this modification the lattice Gale duality.

In Section~\ref{s5} we prove Theorem~\ref{tmain}. Along with the proof we give an explicit description of maximal strongly regular fans, see Corollary~\ref{cormsrf}. Section~\ref{s5-1} contains first properties and examples of $S$-homogeneous varieties and strongly regular fans.
Among others, we describe $S$-homogeneous toric varieties $X$ with $\Cl(X)=\ZZ$ and characterize strongly regular fans consisting of one-dimensional cones.

Section~\ref{s6} describes non-maximal strongly regular fans in terms of corresponding admissible collections. A natural class of non-degenerate homogeneous toric varieties form toric varieties homogeneous under semisimple group. Such varieties are classified in~\cite{AG}. In Section~\ref{s7} we characterize this class in terms of admissible collections.

In Section~\ref{s8} we show that any non-degenerate homogeneous toric variety is a big open toric subset of a maximal $S$-homogeneous toric variety. This implies that every homogeneous toric variety is quasiprojective. We conjecture that any non-degenerate homogeneous toric variety is $S$-homogeneous.


\section{Polyhedral fans and toric varieties} \label{s1}

In this section, we recall basic facts on the correspondence between rational polyhedral fans and toric varieties. For more details, we refer to~\cite{CLS,Fu,Oda}.

By a lattice $N$ we mean a free finitely generated abelian group. We consider the dual lattice $M:=\Hom(N,\ZZ)$ and the associated rational vector spaces $N_{\QQ}:=N\otimes_{\ZZ}\QQ$ and
$M_{\QQ}:=M\otimes_{\ZZ}\QQ$.

A cone in a lattice $N$ is a convex polyhedral cone in the space $N_{\QQ}$. If $\tau$ is a face of a cone~$\sigma$, we write $\tau\preceq\sigma$. We denote by $\sigma(k)$ the set of all $k$-dimensional faces of $\sigma$.

One-dimensional faces of a strictly convex cone are called rays. The primitive vectors of a strictly convex cone $\sigma$ in the lattice $N$ are the primitive lattice vectors on its rays.
A~strictly convex cone in $N$ is called \emph{regular} if the set of its primitive vectors can be supplemented to a basis of the lattice $N$.

A \emph{fan} in a lattice $N$ is a finite collection $\Sigma$ of strictly convex cones in $N$ such that for every $\sigma\in\Sigma$ all faces of $\sigma$ belong to $\Sigma$, and for every
$\sigma_i\in\Sigma$, $i=1,2$, we have $\sigma_1\cap\sigma_2\preceq\sigma_i$. We denote by $\Sigma(k)$ the set of all $k$-dimensional cones in $\Sigma$.
Let $|\Sigma|$ denote the \emph{support} of a fan $\Sigma$, that is the union of all cones in $\Sigma$. A fan is called \emph{regular} if all its cones are regular.

A \emph{toric variety} is a normal algebraic variety $X$ containing an algebraic torus $T$ as an open subset such that the left multiplication on $T$ can be extended to a regular action $T\times X\to X$.

Let $\Sigma$ be a fan in a lattice $N$. For every cone $\sigma\in\Sigma$ we define an affine toric variety $X_{\sigma}:=\Spec(\KK[\sigma^{\vee}\cap M])$, where $\sigma^{\vee}$ is the dual cone in $M_{\QQ}$ to the cone $\sigma$. Gluing together all varieties $X_{\sigma}$ along their isomorphic
open subsets one obtains a toric variety $X_{\Sigma}$. Conversely, any toric variety comes from some fan $\Sigma$ is the lattice $N$ of one-parameter subgroups of the acting torus $T$. The dual lattice $M$ may be interpreted as the lattice of characters of the torus $T$. For every $m\in M$, we denote by $\chi^m$ the corresponding character $T\to\KK^{\times}$.

It is well known that a toric variety $X_{\Sigma}$ is smooth if and only if the fan $\Sigma$ is regular. Further, $X_{\Sigma}$ is complete if and only if the fan $\Sigma$ is complete, that is
$|\Sigma|=N_{\QQ}$.

A toric variety $X_{\Sigma}$ is \emph{degenerate} if it is equivariantly isomorphic to the product of a nontrivial torus $T_0$ and a toric variety of smaller dimension $X_0$. By~\cite[Proposition~3.3.9]{CLS}, $X_{\Sigma}$ is degenerate if and only if there is an invertible non-constant regular function on $X_{\Sigma}$ or, equivalently, the rays in $\Sigma(1)$ do not span the space $N_{\QQ}$. A variety $X_{\Sigma}$ is homogeneous if and only if $X_0$ is homogeneous. So we assume further that $X_{\Sigma}$ is non-degenerate.


\section{Demazure roots and strongly regular fans} \label{s2}

Let $\Sigma$ be a fan in the space $N_{\QQ}$. We denote by $n_{\rho}$ the primitive lattice vector on a ray $\rho\in\Sigma(1)$. Let $N\times M\to\ZZ$, $(n,e)\to\langle n,e\rangle$ be the pairing of the dual lattices $N$ and $M$. For $\rho\in\Sigma(1)$ we consider the set $\RRR_{\rho}$ of all vectors $e\in M$ such that
\begin{enumerate}
\item[(R1)]
$\langle n_{\rho},e\rangle=-1\,\,\mbox{and}\,\, \langle n_{\rho'},e\rangle\geqslant0
\,\,\,\,\forall\,\rho'\in \sigma(1), \,\rho'\ne\rho$;
\smallskip
\item[(R2)]
if $\sigma$ is a cone in $\Sigma$ and $\langle v,e\rangle=0$ for all $v\in\sigma$, then the cone generated by $\sigma$ and $\rho$ is in $\Sigma$ as well.
\end{enumerate}
Note that condition~$(R1)$ implies condition~$(R2)$ if the support $|\Sigma|$ is convex.

The elements of the set $\RRR:=\bigsqcup\limits_{\rho\in\Sigma(1)}\RRR_{\rho}$ are called the \emph{Demazure roots} of the fan $\Sigma$, cf.~\cite[Definition~4]{De} and \cite[Section~3.4]{Oda}. If $e\in\RRR_{\rho}$ then $\rho$ is called the \emph{distinguished ray} of a root~$e$.

Let $X=X_{\Sigma}$ be a toric variety corresponding to the fan $\Sigma$. Denote by $\GG_a$ the additive group of the ground field $\KK$. It is well known that elements of $\RRR$ are in bijection with $\GG_a$-actions on $X$ normalized by the acting torus $T$, see~\cite[Th\'eoreme~3]{De} and \cite[Proposition~3.14]{Oda}. Let us denote the $\GG_a$-subgroup of $\Aut(X)$ corresponding to a root $e$ by $H_e$. Let $\rho_e$ be the distinguished ray corresponding to a root $e$, $n_e$ the primitive lattice vector on $\rho_e$, and $R_e$ the one-parameter subgroup of $T$ corresponding to $n_e$.

There is a bijection between cones $\sigma\in\Sigma$ and $T$-orbits $\OOO_{\sigma}$ on $X$ such that $\sigma_1\subseteq\sigma_2$ if and only if $\OOO_{\sigma_2}\subseteq\overline{\OOO_{\sigma_1}}$. Here $\dim\OOO_{\sigma}=\dim X -\dim\langle\sigma\rangle$.

\smallskip

The following proposition describes the action of the group $H_e$ on $X$. The proof can be found, for example, in \cite[Proposition~5]{AK}.

\begin{proposition} \label{connect}
For every point $x\in X\setminus X^{H_e}$ the orbit $H_e\cdot x$ meets exactly two $T$-orbits $\OOO_1$ and $\OOO_2$ on $X$ with $\dim\OOO_1=\dim\OOO_2+1$. The intersection
$\OOO_2\cap H_e\cdot x$ consists of a~single point, while
$$
\OOO_1\cap H_e\cdot x=R_e\cdot y \quad \text{for any} \quad y\in\OOO_1\cap H_e\cdot x.
$$
\end{proposition}

A pair of $T$-orbits $(\OOO_1,\OOO_2)$ on $X$ is said to be \emph{$H_e$-connected} if
$H_e\cdot x\subseteq \OOO_1\cup\OOO_2$ for some $x\in X\setminus X^{H_e}$. By Proposition~\ref{connect}, we have $\OOO_2\subseteq\overline{\OOO_1}$ and $\dim\OOO_1=\dim\OOO_2+1$. Since the torus $T$ normalizes the subgroup $H_e$, any point of
$\OOO_1\cup\OOO_2$ can actually serve as a point $x$.

We say that a cone $\sigma_2$ in a fan $\Sigma$ is \emph{connected} with its facet $\sigma_1$ by a root $e\in\RRR$ if $e|_{\sigma_2}\le 0$ and $\sigma_1$ is given by the equation $\langle \cdot,e\rangle=0$ in $\sigma_2$.

\begin{lemma} \label{ule} \cite[Lemma~1]{AK}
A pair of $\TT$-orbits $(\OOO_{\sigma_1},\OOO_{\sigma_2})$ is $H_e$-connected if and only if
$\sigma_1$ is a facet of $\sigma_2$ and $\sigma_2$ is connected with $\sigma_1$ by the root $e$.
\end{lemma}

\begin{definition}
A fan $\Sigma$ is called \emph{strongly regular} if every nonzero cone $\sigma\in\Sigma$ is connected with some of its facets by a root.
\end{definition}

We denote by $S(X)$ the subgroup of $\Aut(X)$ generated by subgroups $H_e$, $e\in\RRR$. Let $G(X)$ be the subgroup of $\Aut(X)$ generated by $T$ and $S(X)$. A toric variety $X$ is said to be $S$-homogeneous if the group $S(X)$ acts on $X$ transitively.

\begin{proposition} \label{psh}
A non-degenerate toric variety $X_{\Sigma}$ is $S$-homogeneous if and only if the fan $\Sigma$ is strongly regular.
\end{proposition}

Let us begin with a simpler observation.

\begin{lemma} \label{lll}
A fan $\Sigma$ is strongly regular if and only if the group $G(X)$ acts on~$X$ transitively.
\end{lemma}

\begin{proof}
Without loss of generality we may assume that $X$ is non-degenerate. Suppose that the fan $\Sigma$ is strongly regular. By Lemma~\ref{ule}, every point from a non-open $T$-orbit in $X$ can be sent by some subgroup $H_e$ to a $T$-orbit of higher dimension. It shows that every point on $X$ can be sent by an element of $S(X)$ to an open $T$-orbit, and thus the group $G(X)$ acts transitively on $X$.

Conversely, suppose that the fan $\Sigma$ is not strongly regular. Let $\sigma$ be a nonzero cone in $\Sigma$ which is not connected with any its facet by a root. By Lemma~\ref{ule}, the image of the orbit $\OOO_{\sigma}$ under the action of any root subgroup $H_e$ is contained in its closure $\overline{\OOO_{\sigma}}$. Hence the closure is invariant under the group $S(X)$. This implies that $\overline{\OOO_{\sigma}}$ is a proper $G(X)$-invariant subset in $X$, a contradiction.
\end{proof}

\begin{proof}[Proof of Proposition~\ref{psh}] It remains to show that the group $S(X)$ acts on $X$ transitively for every non-degenerate toric variety with a strongly regular fan $\Sigma$. Let $\rho_1,\ldots,\rho_r$ be the rays of~$\Sigma$. We denote by $e_i$ a root connecting the ray $\rho_i$ with its (unique) facet $\{0\}$. By~Proposition~\ref{connect}, the orbits of the root subgroup $H_{e_i}$ intersected with the open $T$-orbit on $X$ coincide with the orbits of the one-parameter subtorus $R_{e_i}$ represented by the vectors $n_i:=n_{e_i}$ in the lattice $N$.
Since $X$ is non-degenerate, the collection of vectors $n_1,\ldots,n_r$ has full rank in $N$. Thus the open $T$-orbit on $X$ is contained in one $S(X)$-orbit. Containing an open subset on $X$,
this $S(X)$-orbit is $T$-invariant. Lemma~\ref{lll} implies that such an orbit coincides with $X$.
\end{proof}

\begin{corollary}
Every strongly regular fan is regular.
\end{corollary}

\begin{proof}
It follows from the fact that every homogeneous variety is smooth.
\end{proof}

\begin{example} \label{ex1}
The only non-degenerate smooth affine toric variety is the affine space $\AA^n$. Clearly, this variety is $S$-homogeneous, so a regular cone together with all its faces is a strongly regular fan.
\end{example}

\begin{example} \label{ex2}
The automorphism group $\Aut(X)$ of a complete toric variety $X$ is a linear algebraic group,
see~\cite{De,Cox,Nill}. It implies that $X$ is homogeneous if and only if $X$ is $S$-homogeneous.
It is well known that the only homogeneous complete toric varieties are products of projective spaces $\PP^{n_1}\times\ldots\times\PP^{n_m}$, cf. \cite[Theorem~3.9]{Ba}. This implies that complete strongly regular fans are precisely the products of fans of projective spaces.
\end{example}

\begin{remark}
It turns out that properties of regular fans and strongly regular fans are rather different. For example, any subfan of a regular fan is regular, and for strongly regular fans this is not the case. At the same time, there exists at most one maximal strongly regular fan on a given set of rays (see Proposition~\ref{propsuit} below), while some sets of rays can give rise to several maximal (e.g. complete) regular fans.
\end{remark}

\begin{remark}
For an $S$-homogeneous toric variety $X$, the groups $S(X)$ and $G(X)$ may coincide and may not. For example, for the toric variety $X=\PP^n$ the full automorphism group $\PGL(n+1)$ is generated by root subgroups, while for $X=\AA^n$ root subgroups preserve the volume form on $\AA^n$ and the acting torus $T$ does not.
\end{remark}

In general, the subgroup $S(X)$ of the automorphism group $\Aut(X)$ generated by root subgroups may be relatively small. Following~\cite{AKZ}, let us denote by $\SAut(X)$ the subgroup of $\Aut(X)$ generated by all $\GG_a$-subgroups in $\Aut(X)$. The following (non-toric) example shows that the groups $\SAut(X)$ and $S(X)$ may not coincide.

\begin{example}
Let $X$ be an affine variety $\Spec(A)$, where
$$
A=\KK[x,y,z,u,w]/(x+x^2y+z^2+u^3).
$$
Consider a one-dimensional torus action
$$
t\cdot(x,y,z,u,w)=(x,y,z,u,tw).
$$
Denote by $S(X)$ the subgroup of $\Aut(X)$ generated by $\GG_a$-subgroups normalized by the torus.
It is shown in~\cite[Example~3.2]{Li} that any $S(X)$-orbit on $X$ is contained in a subvariety $x=\text{const}$. At the same time, the result of~\cite{Du} implies that there is no non-constant invariant regular function for the action of $\SAut(X)$ on $X$.
\end{example}


\section{Linear Gale duality} \label{s3}
In this section we follow the presentation in \cite[Section~2.2.1]{ADHL}, see also~\cite{OP}.
By a \emph{vector configuration} in a vector space $V$ we mean a finite collection of vectors $v_1,\ldots,v_r\in V$ (possibly with repetitions) that spans the space $V$. A vector configuration $\VVV=\{v_1,\ldots,v_r\}$ in a rational vector space $V$ and a vector configuration $\WWW=\{w_1,\ldots,w_r\}$ in a rational vector space $W$ are \emph{Gale dual} to each other if the following conditions hold:
\begin{enumerate}
\item[(i)]
We have $v_1\otimes w_1+\ldots+v_r\otimes w_r=0$ in $V\otimes W$.
\item[(ii)]
For any rational vector space $U$ and any vectors $u_1,\ldots,u_r\in U$ with
$v_1\otimes u_1+\ldots+v_r\otimes u_r=0$ in $V\otimes U$, there is a unique linear map
$\psi\colon W\to U$ with $\psi(w_i)=u_i$ for $i=1,\ldots,r$.
\item[(iii)]
For any rational vector space $U$ and any vectors $u_1,\ldots,u_r\in U$ with
$u_1\otimes w_1+\ldots+u_r\otimes w_r=0$ in $U\otimes W$, there is a unique linear map
$\phi\colon V\to U$ with $\psi(v_i)=u_i$ for $i=1,\ldots,r$.
\end{enumerate}

If we fix the first configuration in a Gale dual pair, then the second one is determined up to isomorphism. Therefore one configuration is called the \emph{Gale transform} of the other.

Consider vector configurations $\VVV=\{v_1,\ldots,v_r\}$ and $\WWW=\{w_1,\ldots,w_r\}$ in vector spaces $V$ and $W$ respectively, and let $V^*$ be the dual vector space of $V$. Then Gale duality of $\VVV$ and $\WWW$ is characterized by the following property: For any tuple $(a_1,\ldots,a_r)\in\QQ^r$ one has
$$
a_1w_1+\ldots+a_rw_r=0 \ \Longleftrightarrow \ l(v_i)=a_i \ \text{for} \ i=1,\ldots,r \
\text{with some} \ l\in V^*.
$$
Let us present a construction which produces the Gale dual for a configuration $\VVV=\{v_1,\ldots,v_r\}$ in a space $V$. Take the vector space $\QQ^r$ and consider the surjective linear map $\alpha\colon\QQ^r\to V$ given on the standard basis $e_1,\ldots,e_r$ in $\QQ^r$ by $\alpha(e_i)=v_i$, $i=1,\ldots,r$. Consider two mutually dual short exact sequences
$$
\xymatrix{
0
\ar@{->}[rr]
&&
\Ker(\alpha)
\ar@{->}[rr]
&&
\QQ^r
\ar@{->}[rr]^{\alpha}
&&
V
\ar@{->}[rr]
&&
0
\\
0
\ar@{<-}[rr]
&&
(\Ker(\alpha))^*
\ar@{<-}[rr]^{\beta}
&&
(\QQ^r)^*
\ar@{<-}[rr]
&&
V^*
\ar@{<-}[rr]
&&
0
}
$$
Let $e_1^*,\ldots,e_r^*$ be the dual basis in $(\QQ^r)^*$.
Setting $W=(\Ker(\alpha))^*$ and $w_i=\beta(e_i^*)$ for $i=1,\ldots,r$, we obtain the Gale dual configuration $\WWW=\{w_1,\ldots,w_r\}$.

We finish this section with a variant of the separation lemma, cf. \cite[Lemma 4.3]{BH} or \cite[Lemma~2.2.3.2]{ADHL}. Let $\VVV$ be a vector configuration in a rational vector space $V$.
Denote by $\rho_i$ the ray in $V$ spanned by the vector $v_i$ from $\VVV$.
Consider two strictly convex polyhedral cones $\sigma$ and $\sigma'$ in $V$ with
$$
\sigma(1)=\{\rho_i, i\in I\} \quad \text{and} \quad \sigma'(1)=\{\rho_j, j\in J\}.
$$

\begin{lemma} \label{sep}
Let $(W,\WWW)$ be the linear Gale transform of $(V,\VVV)$. Then the intersection of the cones $\sigma$ and $\sigma'$ is a face of each of them if and only if the cones $\cone(w_k, k\notin I)$ and $\cone(w_s, s\notin J)$ in the space $W$ have a common interior point.
\end{lemma}

\begin{proof}
The intersection of the cones $\sigma$ and $\sigma'$ is a face of each of them if and only if
there is a linear function $l\in V^*$ such that
$$
l(v_i)\ge 0 \quad \text{for all} \quad i\in I, \quad
l(v_j)\le 0 \quad \text{for all} \quad j\in J,
$$
and for any $s\in I\cup J$ we have $l(v_s)=0$ if and only if $s\in I\cap J$. This condition means that there is a relation
$$
\sum_{i\in I\setminus J} \alpha_iw_i - \sum_{j\in J\setminus I} \beta_j w_j +\sum_{t\notin I\cup J} \gamma_tw_t=0
$$
with some positive rational coefficients $\alpha_i, \beta_j$ and some rational coefficients $\gamma_t$. This relation is equivalent to
$$
\sum_{k\notin I} \mu_kw_k=\sum_{s\notin J} \nu_sw_s
$$
with some positive rational coefficients $\mu_k$ and $\nu_s$. The latter relation means that the cones $\cone(w_k, k\notin I)$ and $\cone(w_s, s\notin J)$ have a common interior point.
\end{proof}


\section{Lattice Gale transform} \label{s4}
A \emph{vector configuration} $\NNN$ in a lattice $N$ is a finite collection of vectors $n_1,\ldots,n_r\in N$ that spans the vector space $N_{\QQ}$. Consider the lattice $\ZZ^r$ with the standard basis $e_1,\ldots,e_r$ and the exact sequence
$$
\xymatrix{
0
\ar@{->}[rr]
&&
L
\ar@{->}[rr]
&&
\ZZ^r
\ar@{->}[rr]^{\alpha}
&&
N
}
$$
defined by $\alpha(e_i)=n_i$, $i=1,\ldots,r$. Let us identify the dual lattice of $\ZZ^r$
with $\ZZ^r$ using the dual basis $e_1^*,\ldots,e_r^*$. Let $M:=\Hom(N,\ZZ)$. The homomorphism
$M\to\ZZ^r$ dual to $\alpha$ gives rise to the short exact sequence of abelian groups
$$
\xymatrix{
\ \ \ \ \ \ \ \ \ \ \ \ \ \ \ \ \ \ \ \ 0
\ar@{<-}[rr]
&&
P
\ar@{<-}[rr]^{\beta}
&&
\ZZ^r
\ar@{<-}[rr]
&&
M
\ar@{<-}[rr]
&&
0 \ \ \ \ \ \ \ \ \ \ \ \ \  (*)
}
$$
Let $a_i=\beta(e_i^*)$ with $i=1,\ldots,r$. By construction, the vectors $a_1,\ldots,a_r$ generate the group~$P$. We call the collection $\AAA=\{a_1,\ldots,a_r\}$ the \emph{lattice Gale transform} of the configuration~$\NNN$. Replacing all groups in these sequences by their tensor products with $\QQ$, we obtain the linear Gale duality considered above.

Conversely, given elements $a_1,\ldots,a_r$ that generate a group $P$, we can reconstruct sequence~$(*)$, the lattice $N=\Hom(M,\ZZ)$, the dual homomorphism $\ZZ^r\to N$ and thus the vectors $n_1,\ldots,n_r$.

\begin{remark} \label{remus}
Let $\NNN=\{n_1,\ldots,n_r\}$ be a vector configuration in a lattice $N$. The vector $n_1$ is a primitive vector in $N$ if and only if the vectors $a_2,\ldots,a_r$ generate the group $P$. Indeed, $n_1$ is primitive if and only if there is an element $e\in M$ such that $\langle n_1,e\rangle=\pm 1$, or, equivalently, there is a relation $a_1+\alpha_2a_2+\ldots+\alpha_ra_r=0$ with some integer $\alpha_2,\ldots,\alpha_r$.

More generally, a subset $n_i, i\in I$ can be supplemented to a basis of $N$ if and only if for any $i\in I$ the element $a_i$ lies in the subgroup generated by $a_j, j\notin I$.
\end{remark}

\begin{example}
The lattice Gale transform of the configuration $\NNN=\{n_1,n_2\}$ in $N=\ZZ^2$ with $n_1=(1,0)$ and $n_2=(1,2)$ is the collection $\AAA=\{a_1,a_2\}$ in the group $P=\ZZ/2\ZZ$ with $a_1=a_2=\overline{1}$. At the same time, the linear Gale transform of the configuration $\VVV=\{v_1,v_2\}$ in $V=\QQ^2$ with $v_1=(1,0)$ and $v_2=(1,2)$ is the collection $\WWW=\{0,0\}$ in the space $W=\{0\}$.
\end{example}

Now we are going to establish a relation between the lattice Gale duality and Demazure roots.

\begin{definition} \label{defsuit}
A vector configuration $\NNN=\{n_1,\ldots,n_r\}$ in a lattice $N$ is called \emph{suitable} if
for any $i=1,\ldots,r$ there exists a vector $e_i\in\Hom(N,\ZZ)$ such that
$\langle n_i,e_i\rangle=-1$ and $\langle n_j,e_i\rangle\ge 0$ for all $j\ne i$.
\end{definition}

We recall that a collection $\AAA=\{a_1,\ldots,a_r\}$ of elements (possibly with repetitions) of an abelian group $P$ is \emph{admissible} if $\AAA$ generates the group $P$ and for any $a_i\in\AAA$ the element $a_i$ is contained in the semigroup generated by $\AAA\setminus\{a_i\}$.

\begin{lemma} \label{suit-adm}
A vector configuration $\NNN=\{n_1,\ldots,n_r\}$ in a lattice $N$ is suitable if and only if its lattice Gale transform $\AAA$ in $P$ is an admissible collection.
\end{lemma}

\begin{proof}
An element $a_i\in\AAA$ is contained in the semigroup generated by $\AAA\setminus\{a_i\}$ if and only if we have $a_i=\sum_{j\ne i}\alpha_ja_j$ for some non-negative integers $\alpha_j$. The latter condition means that there exists an element $e_i\in\Hom(N,\ZZ)$ with
$$
\langle n_i,e_i\rangle=-1 \quad \text{and} \quad \langle n_j,e_i\rangle=\alpha_j \quad
\text{for all} \quad j\ne i.
$$
\end{proof}


\section{Proof of Theorem~\ref{tmain}} \label{s5}

We begin this section with some preliminary results.  A collection of rays $\rho_1,\ldots,\rho_r$ in the space $N_{\QQ}$ is said to be \emph{suitable} if the set of primitive lattice vectors on these rays is a suitable vector configuration.

\begin{lemma} \label{lemsuit}
For a strongly regular fan $\Sigma$, the collection of rays $\Sigma(1)$ is suitable.
\end{lemma}

\begin{proof}
By definition of a strongly regular fan, every ray $\rho_i$ is connected with its facet $\{0\}$ by a root $e_i$. Then the vector $e_i$ satisfies the conditions of Definition~\ref{defsuit}.
\end{proof}

\begin{definition}
A strongly regular fan $\Sigma$ is \emph{maximal} if it cannot be realized as a proper subfan of a strongly regular fan $\Sigma'$ with $\Sigma'(1)=\Sigma(1)$.
\end{definition}

\begin{proposition} \label{propsuit}
For every suitable collection of rays $\rho_1,\ldots,\rho_r$ in $N_{\QQ}$ there exists a unique maximal strongly regular fan $\Sigma$ with $\Sigma(1)=\{\rho_1,\ldots,\rho_r\}$.
\end{proposition}

\begin{proof}
Let $\Omega$ be the set of strictly convex polyhedral cones $\sigma$ in $N_{\QQ}$ with $\sigma(1)\subseteq\{\rho_1,\ldots,\rho_r\}$. With every $\sigma\in\Omega$ one associates a subset
$I\subseteq\{1,\ldots,r\}$ as $\sigma(1)=\{\rho_i, i\in I\}$.

Let $\AAA=\{a_1,\ldots,a_r\}$ be the lattice Gale transform of the vector configuration $\NNN=\{n_1,\ldots,n_r\}$. Denote by $\Gamma(\sigma)$ the semigroup in $P$ generated by $a_j$, $j\notin I$. In particular, we have $\Gamma(\{0\})=A$, where $A$ is the semigroup generated by $\AAA$.

Let
$$
\Sigma=\Sigma(P,\AAA):=\{\sigma \in\Omega \ ; \ \Gamma(\sigma)=A\}.
$$

We have to check four assertions.

\begin{enumerate}
\item[(A1)]
$\Sigma$ is a fan and $\Sigma(1)=\{\rho_1,\ldots,\rho_r\}$.
\item[(A2)]
The fan $\Sigma$ is strongly regular.
\item[(A3)]
The fan $\Sigma$ is maximal.
\item[(A4)]
Every strongly regular fan $\hat{\Sigma}$ with $\hat{\Sigma}(1)=\{\rho_1,\ldots,\rho_r\}$ is a subfan of $\Sigma$.
\end{enumerate}

We start with (A1). By definition, if $\tau$ is a face of a cone $\sigma$ from $\Omega$, then
$\tau$ is in $\Omega$ and $\Gamma(\sigma)$ is contained in $\Gamma(\tau)$. In particular, if
$\Gamma(\sigma)=A$ then $\Gamma(\tau)=A$ as well. This shows that a face of a cone from $\Sigma$ is contained in $\Sigma$.

We have to check that the intersection of two cones from $\Sigma$ is a face of each of them. This follows from Lemma~\ref{sep}.

\smallskip

We proceed with (A2). Let $\sigma\in\Sigma$ and $\tau$ be a facet of $\sigma$. We take $\rho_i\in\sigma(1)\setminus\tau(1)$. Assume that $\sigma(1)=\{\rho_k, k\in I\}$ for a subset $I$ in $\{1,\ldots,r\}$. Since $\Gamma(\sigma)=A$, we have
$$
a_i=\sum_{j\in\{1,\ldots,r\}\setminus I} \alpha_ja_j \quad \text{with some} \quad
\alpha_j\in\ZZ_{\ge 0}.
$$
It means that there is a vector $e\in\Hom(N,\ZZ)$ with
$$
\langle n_i,e\rangle=-1, \quad \langle n_j, e\rangle\ge 0 \quad \text{for all} \quad j\ne i, \quad
\text{and} \quad \langle n_k, e\rangle=0 \quad \text{for all} \quad k\in I\setminus\{i\}.
$$
In particular, all rays of the cone $\sigma$ except for $\rho_i$ lie in the hyperplane
$\langle\cdot,e\rangle=0$ and thus we have $\sigma(1)=\tau(1)\cup\{\rho_i\}$.

We still have to prove that the element $e$ is a Demazure root of the fan $\Sigma$. Condition $(R1)$ obviously holds. Let us check condition $(R2)$. Let $\sigma'\in\Sigma$ and $e|_{\sigma'}=0$. We have to show that the cone $\sigma''=\cone(\sigma',\rho_i)$ is in $\Sigma$. The condition
$\sigma'\in\Sigma$ means that the elements $a_s$ with $\rho_s\notin\sigma'(1)$ generate the semigroup $A$. The condition $e|_{\sigma'}=0$ implies that the element $a_i$ is a non-negative integer linear combination of the elements $a_k$ with $\rho_k\notin\sigma'(1)$ and $k\ne i$. This shows that the elements $a_k$ generate the semigroup $A$ as well, thus $\Gamma(\sigma'')=A$ and
$\sigma''\in\Sigma$.

We conclude that any nonzero cone in $\Sigma$ is connected by a root with any its facet, and the fan $\Sigma$ is strongly regular.

\smallskip

We come to (A3). Assume that we can add to the fan $\Sigma$ some cones $\sigma_1,\ldots,\sigma_m$ from $\Omega$ and obtain a strongly regular fan $\Sigma'$. For every $\sigma_i$ there is a chain of facets $\{0\}\preceq\ldots\preceq\sigma_i'\preceq\sigma_i$ connected by roots of the fan $\Sigma'$. Hence we have $$
\Gamma(\sigma_i)=\Gamma(\sigma_i')=\ldots=\Gamma(\{0\})=A
$$
and $\sigma_i\in\Sigma$, a contradiction.
\smallskip

Finally we prove assertion (A4). Let $\Sigma'$ be a strongly regular fan with
$\Sigma'(1)=\{\rho_1,\ldots,\rho_r\}$. Then for any $\sigma\in\Sigma'$ we again have a chain of facets $\{0\}\preceq\ldots\preceq\sigma'\preceq\sigma$ connected by roots of the fan $\Sigma'$.
This implies $\Gamma(\sigma)=A$ and thus $\Sigma'$ is contained in $\Sigma$.

This completes the proof of Proposition~\ref{propsuit}.
\end{proof}

\begin{corollary} \label{cormsrf}
Every maximal strongly regular fan has the form
$$
\Sigma(P,\AAA):=\{\sigma \in\Omega \ ; \ \Gamma(\sigma)=A\}
$$
for some abelian group $P$ and some admissible collection $\AAA$ of elements in $P$.
\end{corollary}

\begin{corollary}
Let $\Sigma$ be a maximal strongly regular fan and $\sigma$ a nonzero cone in $\Sigma$. Then $\sigma$ is connected with any its facet by a root of the fan $\Sigma$.
\end{corollary}

\begin{proof}
The statement follows from Corollary~\ref{cormsrf} and the proof of (A2) in the proof of Proposition~\ref{propsuit}.
\end{proof}

\begin{proof}[Proof of Theorem~\ref{tmain}]
By Proposition~\ref{psh}, $S$-homogeneous toric varieties correspond to strongly regular fans. In turn, maximal $S$-homogeneous toric varieties correspond to maximal strongly regular fans. Lemma~\ref{lemsuit} and Proposition~\ref{propsuit} show that maximal strongly regular fans are in bijection with suitable collections of rays or, equivalently, with suitable vector configurations
$\NNN$ in a lattice $N$. By Lemma~\ref{suit-adm}, the lattice Gale transform establishes a bijection between suitable vector configurations $(N,\NNN)$ and admissible collections $(P,\AAA)$. It~remains to notice that all fans, vector configurations and collections above are defined up to isomorphism of the lattice $N$ and of the group $P$, respectively. So maximal $S$-homogeneous toric varieties correspond to equivalence classes of pairs $(P,\AAA)$.
\end{proof}

\begin{remark}
Let us recall why for an $S$-homogeneous toric variety $X$ the group $P$ constructed above can be interpreted as the divisor class group $\Cl(X)$ and the collection $\AAA$ is the collection of classes $[D_1],\ldots,[D_r]$ of $T$-invariant prime divisors on $X$. It is well known that $T$-invariant prime divisors on $X$ are in bijection with rays of the fan $\Sigma_X$, their classes generate the group $\Cl(X)$, and the defining relations for this generating system are of the form
$$
\langle n_1,e\rangle[D_1]+\ldots+\langle n_r,e\rangle[D_r]=0,
$$
where $e$ runs through the lattice $M$, see e.g.~\cite[Section~3.4]{Fu}. This coincides with the definition of the lattice Gale transform of the vector configuration $\{n_1,\ldots,n_r\}$. Moreover, since any effective Weil divisor on a toric variety is linearly equivalent to a $T$-invariant effective Weil divisor, the semigroup $A$ is the semigroup of classes of effective Weil divisors on $X$.
\end{remark}


\section{First properties and examples of strongly regular fans} \label{s5-1}

Let us list some basic observations on maximal strongly regular fans $\Sigma(P,\AAA)$ and maximal $S$-homogeneous toric varieties $X(P,\AAA)$ corresponding to an admissible collection $\AAA$ in an abelian group $P$.

\smallskip

\begin{enumerate}
\item[(P1)]
The variety $X(P,\AAA)$ is affine if and only if $P=0$ and $\AAA$ is the element $0$ taken $n$ times for some $n\ge 0$. This follows from Example~\ref{ex1}.

\item[(P2)]
The variety $X(P,\AAA)$ is complete if and only if $P$ is a lattice and there are a basis $e_1,\ldots,e_m$ of $P$ and integers $n_1\ge 2,\ldots,n_m\ge 2$ such that $\AAA=\{a_1 (n_1 \ \text{times}),\ldots,a_m (n_m \ \text{times})\}$. This follows from Example~\ref{ex2}.

\item[(P3)]
The variety $X(P,\AAA)$ is quasiaffine if and only if the cone in the space $P_{\QQ}=P\otimes_{\ZZ}\QQ$ generated by the vectors $a\otimes 1$, $a\in\AAA$, coincides with $P_{\QQ}$. Such a variety is the regular locus $X^{\reg}$ of a non-degenerate affine toric variety $X$, cf. \cite[Theorem~2.1]{AKZ}.

\item[(P4)]
If $P=P_1\oplus P_2$ and $\AAA=\AAA_1\oplus\AAA_2$, then $X(P,\AAA)\cong X(P_1,\AAA_1)\times X(P_2,\AAA_2)$.
\end{enumerate}

\smallskip

\begin{example}
Let $P=\ZZ/3\ZZ$ and $\AAA=\{\overline{1},\overline{1}\}$. Then $X(P,\AAA)=X(\sigma)^{\reg}$ with $\sigma=\cone((1,0),(2,3))$. If $\AAA=\{\overline{1},\overline{2}\}$ then $X(P,\AAA)=X(\sigma)^{\reg}$ with $\sigma=\cone((1,0),(1,3))$.
\end{example}

The technique developed in this paper allows to obtain explicit classification results. As an illustration, let us classify maximal $S$-homogeneous toric varieties $X$ with $\dim X=d$ and $\Cl(X)=\ZZ$. To do this, we need to find all admissible collections $\AAA$ in the group $\ZZ$.
We divide all such collections into three types.

\smallskip

{\it Type 1}. The collection $\AAA$ contains both positive and negative elements. Here we have  $X=X(\sigma)^{\reg}$, where $\sigma$ is a strictly convex polyhedral cone with
$d+1$ rays in $\AA^d$.

\smallskip

{\it Type 2}. All elements in $\AAA$ are positive. Consider the weighted projective space
$Z=\PP(a_1,\ldots,a_r)$, see~\cite[Section~2.0]{CLS} for precise definition. Clearly,
the variety $X$ is a smooth open toric subset in $Z$. Using Remark~\ref{remus}, one can check that $X$ coincides with $Z^{\reg}$ if and only if for every subcollection $\AAA'\subseteq\AAA$ that generates the group $\ZZ$, the semigroup generated by $\AAA'$ equals $A$.

\smallskip

{\it Type 3}. All elements in $\AAA$ are non-negative and $\AAA$ contains $0$. In this case $X$ is a direct product of an affine space and a variety of Type~2 with smaller dimension.

\smallskip

\begin{example}
Let us classify varieties of Type~2 for $d=3$. We have two possibility for the variety $Z$.

\begin{enumerate}
\item[1)]
$Z=\PP(1,1,a_3,a_4)$ with some $a_3,a_4\in\ZZ_{>0}$.

\item[2)]
$Z=\PP(a_1,a_1,a_2,a_2)$ with $1<a_1<a_2$ and $(a_1,a_2)=1$.
\end{enumerate}

In the second case we have $X=Z^{\reg}$, while in the first one this is not always true. For instance, with $Z=\PP(1,1,2,3)$ the subset $Z^{\reg}\setminus X$ is an irreducible curve.

\end{example}

The last observation concerns strongly regular fans composed of one-dimensional cones.

\begin{proposition} \label{1-sceleton}
Let $\Sigma$ be a fan with $\Sigma=\Sigma(1)\cup\{0\}$. Then $\Sigma$ is strongly regular if and only if either $\Sigma=\sigma(1)\cup\{0\}$ for a strictly convex polyhedral cone $\sigma$ or
$\Sigma=\Sigma_{\PP^1}$.
\end{proposition}

\begin{proof}
Let $\Sigma(1)=\{\rho_1,\ldots,\rho_r\}$ and $\NNN=\{n_1,\ldots,n_r\}$ be the corresponding vector configuration in $N$. The fan $\Sigma$ is strongly regular if and only if every $\rho_i$ is connected with $\{0\}$ by a root of $\Sigma$, i.e., there exists an element $e_i\in\Hom(N,\ZZ)$ such that $\langle n_i, e_i\rangle=-1$ and $\langle n_j, e_i\rangle>0$ for all $j\ne i$. Note that the case $\langle n_j, e_i\rangle=0$ for some $j$ is excluded because $\Sigma$ does not contain the cone $\cone(\rho_i,\rho_j)$.

Clearly, the fan $\Sigma_{\PP^1}$ is strongly regular. If $\Sigma=\sigma(1)\cup\{0\}$, then for every $\rho_i$ there is a linear function which is zero on $\rho_i$ and positive on all other rays. It shows that the desired functions $e_i$ exist.

Conversely, assume that $\Sigma$ is strongly regular. The case $r\le 2$ is obvious. So we suppose that $r\ge 3$. Then the linear function $e_1+\ldots+e_r$ is positive on all rays and thus the rays generate a strictly convex cone $\sigma$. Existence of the functions $e_i$ implies that every $\rho_i$ is a ray of $\sigma$.
\end{proof}


\section{Non-maximal $S$-homogeneous toric varieties} \label{s6}

In Section~\ref{s5} we gave an explicit description of maximal strongly regular fans. The aim of this section is to develop our combinatorial language further and to describe strongly regular subfans of a given maximal strongly regular fan.

Let $P$ be an abelian group, $\AAA=\{a_1,\ldots,a_r\}$ an admissible collection of elements in $P$, and $A$ the semigroup in $P$ generated by $\AAA$.

\begin{definition}
A \emph{link} is a pair $(a,\AAA')$, where $\AAA'$ is a subcollection of $\AAA$, $a\in\AAA\setminus\AAA'$, and there exists an expression $a=\sum_j \alpha_ja_j$, where $a_j$ runs through $\AAA'$ and $\alpha_j\in\ZZ_{>0}$.
\end{definition}

We say that a subcollection $\BBB\subseteq\AAA$ is \emph{generating}, if the elements of $\BBB$ generate the semigroup $A$. Let $\GG$ be a set of generating collections in $\AAA$.

\begin{definition}
A link $(a,\AAA')$ is called a \emph{$\GG$-link} if for any $\BBB\in\GG$ the condition $\AAA'\cup\{a\}\subseteq\BBB$ implies $\BBB\setminus\{a\}\in\GG$.
\end{definition}

\begin{definition}
A set $\GG$ of generating collections in $\AAA$ is called \emph{connected} if the following conditions hold:
\begin{enumerate}
\item[(C1)]
$\AAA\setminus\{a_i\}\in\GG$ for any $i=1,\ldots,r$;
\item[(C2)]
$\BBB\in\GG$ and $\BBB\subseteq\BBB'\subseteq\AAA$ implies $\BBB'\in\GG$;
\item[(C3)]
if $\BBB\in\GG$ and $\BBB\ne\AAA$ then there is a $\GG$-link $(a,\AAA')$ with $\AAA'\subseteq\BBB$ and $a\notin\BBB$.
\end{enumerate}
\end{definition}

Let $\{\rho_1,\ldots,\rho_r\}$ be a suitable collection of rays in a space $N_{\QQ}$ and
$\NNN=\{n_1,\ldots,n_r\}$ the corresponding suitable vector configuration in $N$. Consider the lattice Gale transform $(P,\AAA)$ of $(N,\NNN)$.

\begin{proposition} \label{propnm}
Strongly regular fans $\Sigma$ with $\Sigma(1)=\{\rho_1,\ldots,\rho_r\}$ are in bijection with connected sets $\GG$ of generating collections in $\AAA$.
\end{proposition}

\begin{proof}
With any subcollection $\BBB\subseteq\AAA$ we associate a cone $\sigma(\BBB)=\cone(\rho_j, a_j\notin\BBB)$. By Corollary~\ref{cormsrf}, the maximal strongly regular fan $\Sigma(P,\AAA)$ is the set of cones associated with all generating subcollections in $\AAA$,
and any strongly regular fan $\Sigma$ with $\Sigma(1)=\{\rho_1,\ldots,\rho_r\}$ is a subfan of $\Sigma(P,\AAA)$.

Let $\Sigma^{\GG}$ be the set of cones $\sigma(\BBB)$, $\BBB\in\GG$. We are going to check that conditions $(C1)$-$(C3)$ are equivalent to the fact that $\Sigma^{\GG}$ is a strongly regular subfan of $\Sigma(P,\AAA)$.

Condition $(C1)$ means that $\Sigma^{\GG}(1)=\{\rho_1,\ldots,\rho_r\}$. Since all cones in
$\Sigma(P,\AAA)$ are regular, condition~$(C2)$ means that with any cone $\sigma(\BBB)$ the fan $\Sigma$ contains all its faces $\sigma(\BBB')$. We know that in $\Sigma(P,\AAA)$ every two cones meet at a face. Hence conditions $(C1)$-$(C2)$ mean that $\Sigma^{\GG}$ is a subfan of $\Sigma(P,\AAA)$ with $\Sigma^{\GG}(1)=\{\rho_1,\ldots,\rho_r\}$.

Let us show that condition $(C3)$ means that the fan $\Sigma^{\GG}$ is strongly regular. Existence of a $\GG$-link $(a_i,\AAA')$ expresses the fact that there is a root $e$ of the fan
$\Sigma^{\GG}$ such that $\langle n_i,e\rangle=-1$ and the condition $\langle n_j,e\rangle>0$ is equivalent to $a_j\in\AAA'$. Then $(C3)$ means that every nonzero cone $\sigma(\BBB)$ in $\Sigma^{\GG}$ is connected with its facet $\sigma(\BBB\cup\{a\})$ by a root associated with the
corresponding $\GG$-link $(a,\AAA')$.
\end{proof}

\begin{example}
Let $P=\ZZ$ and $\AAA=\{a_1=a_2=a_3=1\}$. Here $\Sigma(P,\AAA)=\Sigma_{\PP^2}$. The set
$\GG=\{\AAA,\AAA\setminus\{a_1\},\AAA\setminus\{a_2\},\AAA\setminus\{a_3\}\}$ corresponds to the subfan $\Sigma=\Sigma(P,\AAA)(1)\cup\{0\}$. Links in this case are precisely the pairs $(a_i,\{a_j\})$, $i\ne j$. None of these links is a $\GG$-link. Thus the fan $\Sigma$ is not strongly regular.
\end{example}

\begin{remark}
In general, Proposition~\ref{1-sceleton} and Property $(P3)$ provide a criterion for the set $\GG=\{\AAA,\AAA\setminus\{a_1\},\ldots,\AAA\setminus\{a_r\}\}$ to be connected.
\end{remark}


\section{Toric varieties homogeneous under semisimple group} \label{s7}

In~\cite{AG}, a classification of toric varieties that are homogeneous under an action of a semisimple linear algebraic group is obtained. Let us present this classification in terms of Proposition~\ref{propnm}.

Consider a quasiaffine variety
$$
\XX =\XX(n_1,\dots,n_m) :=
(\KK^{n_1}\setminus \{0\} )\times\dots\times(\KK^{n_m}\setminus\{0\})
$$
with $n_i\geq 2$. The group $G=G_1\times\ldots\times G_m$, where every component $G_i$ is either $\SL(n_i)$ or $\Sp(n_i)$, and $n_i$ is even in the second case, acts on $\XX$ transitively and effectively. Let $\SS=(\KK^{\times})^m$ be an algebraic torus acting on $\XX$ by componentwise scalar multiplication, and
$$
p \, \colon \, \XX \, \to \ \YY \ := \ \PP^{n_1-1}\times\dots\times\PP^{n_m-1}
$$
be the quotient morphism. Fix a closed subgroup $S\subseteq\SS$. The action of the group $S$ on $\XX$ admits a geometric quotient $p_X \colon\XX\to X := \XX/S$. The variety $X$ is toric, it carries the induced action of the quotient group $\SS/S$, and there is a quotient morphism $p^X \colon X \to \YY$ for this action closing the commutative diagram

$$
\xymatrix{
\XX \ar[dr]_p \ar[rr]^{p_X} & & X \ar[dl]^{p^X} \\
 & \YY & }
$$

The induced action of the group $G$ on $X$ is transitive and locally effective.
We say that the $G$-variety $X$ is obtained from $\XX$ by \emph{central factorization}. By \cite[Theorem~1.1]{AG}, every toric variety with a transitive action of a semisimple group can be obtained this way.

The above diagram of quotient morphisms of homogeneous spaces gives rise to the diagram of homomorphisms of divisor class groups
$$
\xymatrix{
\{0\} & & P \ar[ll] \\
 & \ZZ^m \ar[ul] \ar[ur] & }
$$

It shows that an admissible collection corresponding to the variety $X$ is obtained as the projection of the collection corresponding to the variety $\YY$. This way we obtain

\begin{proposition} \label{prophomss}
Toric varieties $X$ homogeneous under an action of a semisimple linear algebraic group are in bijection with pairs $(P,\AAA)$, where
\begin{enumerate}
\item[1)]
$P$ is an abelian group;
\item[2)]
$\AAA=\{a_1\,(n_1 \, \text{times}),\ldots,a_m\,(n_m \, \text{times})\}$ with some $n_1\ge 2,\ldots,n_m\ge 2$, and the elements $a_1,\ldots,a_m$ generate the group $P$.
\end{enumerate}
A variety $X$ represented by a pair $(P,\AAA)$ is an open toric subset in the variety $X(P,\AAA)$
corresponding to the connected set $\GG$ of generating collections in $\AAA$ such that every collection $\BBB$ in $\GG$ contains at least one element from each of the $m$ groups of elements in $\AAA$.
\end{proposition}

It is easy to see that the variety $X$ coincides with $X(P,\AAA)$ if and only if the elements $\{a_1,\ldots,a_m\}\setminus\{a_i\}$ do not generate the semigroup $A$ generated by $\{a_1,\ldots,a_m\}$ for any $i=1,\ldots,m$.

\begin{example}
Let $P=\ZZ/2\ZZ\oplus\ldots\oplus\ZZ/2\ZZ$ ($m$ times) and $a_i=(\overline{0},\ldots,\overline{1},\ldots,\overline{0})$ with $\overline{1}$ at the $i$th place. Then the variety $X$ is quasiaffine and coincides with $X(P,\AAA)$ for any $n_1,\ldots,n_m$.
\end{example}

\begin{remark}
In \cite[Proposition~4.5]{AG}, one can find a description of the fan $\Sigma^{\GG}$, where $\GG$ is as in Proposition~\ref{prophomss}.
\end{remark}

\begin{remark}
Toric varieties $X_{\Sigma}$, where a fan $\Sigma$ contains some fan $\Sigma^{\GG}$ as a subfan, provide examples of embedding with small boundary of homogeneous spaces of semisimple groups, see~\cite{AH} for details.
\end{remark}

\begin{problem}
Classify toric varieties homogeneous under linear algebraic groups.
\end{problem}


\section{Homogeneous toric varieties} \label{s8}

We recall that a toric variety $X$ is called homogeneous if the automorphism group $\Aut(X)$ acts on $X$ transitively.

\begin{theorem} \label{thom}
Let $X$ be a non-degenerate homogeneous toric variety. Then there exists an open toric embedding
$X\subseteq X'$ into a maximal $S$-homogeneous toric variety $X'$ with
${\codim_{X'}(X'\setminus X)\ge 2}$.
\end{theorem}

\begin{corollary}
Every homogeneous toric variety is quasiprojective.
\end{corollary}

\begin{proof}
It suffices to show that every maximal $S$-homogeneous toric variety is quasiprojective. This follows from Corollary~\ref{cormsrf}, \cite[Corollary~10.3]{BH} and \cite[Theorem~2.2.2.6]{ADHL}.
\end{proof}

We begin the proof of Theorem~\ref{thom} with a preliminary result.

\begin{lemma} \label{tlin}
Let $X$ be a non-degenerate smooth toric variety and $x$ a point on $X$. Consider an effective  divisor $D$ on $X$ whose support does not contain $x$. Then there is a $T$-invariant effective divisor $D'$ which is linearly equivalent to $D$ and whose support does not contain $x$.
\end{lemma}

\begin{proof}
The divisor $D$ defines a line bundle $L\to X$ and $L$ admits a $T$-linearization, see e.g. \cite[Section~4.2.2]{ADHL} for details. In particular, the space of global sections $H^0(X,L)$ carries a structure of a rational $T$-module, and any vector in $H^0(X,L)$ is a sum of $T$-eigenvectors. Sections that represent effective divisors with support not passing through $x$ form a subspace $U$ in $H^0(X,L)$. By assumption, the subspace $U$ is proper. Hence there is a $T$-eigenvector $v$ in $H^0(X,L)\setminus U$. An effective $T$-invariant divisor on $X$ represented by $v$ is the desired divisor $D'$.
\end{proof}

\begin{proof}[Proof of Theorem~\ref{thom}]
Let $X$ be a non-degenerate homogeneous toric variety and $\Sigma$ the associated fan.
Consider the set of rays $\Sigma(1)=\{\rho_1,\ldots,\rho_r\}$, the corresponding vector configuration $\{n_1,\ldots,n_r\}$ in the lattice $N$, and the lattice Gale transform $(P,\AAA)$.

With any point $x\in X$ one associates the semigroup $C(x)$ in $\Cl(X)$ of classes of effective divisors on $X$ whose support does not contain $x$. For a point $x$ in the open $T$-orbit on $X$ the semigroup $C(x)$ coincides with the semigroup $C$ of all classes of effective divisors on $X$. Since $X$ is homogeneous, we have $C(x)=C$ for all points $x\in X$.

By Lemma~\ref{tlin}, the semigroup $C(x)$ equals the semigroup generated by classes of $T$-invariant prime divisors on $X$ which do not pass through $x$. Under identification of the group $P$ with $\Cl(X)$, the semigroup $C(x)$ coincides with the semigroup $\Gamma(\sigma)$, where $\sigma$ is a cone in $\Sigma$ associated with the $T$-orbit of the point $x$. It shows that for a homogeneous toric variety $X$ we have $\Gamma(\sigma)=A$ for every $\sigma\in\Sigma$, where $A$ is the semigroup generated by the collection $\AAA$. In particular, $\Gamma(\rho_i)=A$ for any $i=1,\ldots,r$. It means that the set $\{\rho_1,\ldots,\rho_r\}$ is suitable or, equivalently, the collection $\AAA$ is admissible.

Finally, the condition $\Gamma(\sigma)=A$ implies that the fan $\Sigma$ is a subfan of the fan $\Sigma(P,\AAA)$. Equivalently, $X$ is an open toric subset of the maximal $S$-homogeneous toric variety $X'=X(P,\AAA)$. The condition $\Sigma(1)=\Sigma(P,\AAA)(1)$ implies
${\codim_{X'}(X'\setminus X)\ge 2}$.
\end{proof}

\begin{conjecture} \label{con}
Every non-degenerate homogeneous toric variety is $S$-homogeneous.
\end{conjecture}

In view of Theorem~\ref{thom}, Conjecture~\ref{con} means that every toric variety $X_{\Sigma}$, where $\Sigma$ is a non-strongly regular subfan of a maximal strongly regular fan, is not homogeneous. Computations with low-dimensional toric varieties confirm the conjecture.


\end{document}